\documentclass[12pt,a4paper,
fleqn,reqno]{amsart}                 
\usepackage[T1]{fontenc}                   
\usepackage[utf8]{inputenc}                 
\usepackage[english]{babel}    
 \usepackage{csquotes}
\usepackage{graphicx}                      %
\usepackage[font=small]{quoting}    
\usepackage{float}
\usepackage{caption}  
\usepackage[top=1.2in,bottom=1.8in,left=1.in,right=1.in]{geometry}                  
\usepackage{verbatim}
\usepackage{color}
\usepackage{xcolor}
\usepackage{picture}
\usepackage{amsmath,amsthm,amsfonts,amssymb}
\usepackage{comment}
\usepackage{hyperref}
\usepackage{mathtools}
\hypersetup{colorlinks,linkcolor={black}, citecolor={black},urlcolor={black}}  
\usepackage{enumitem}
\usepackage[url=false,doi=false,isbn=false,giveninits=true,   
style=numeric,maxbibnames=99]{biblatex}
\newtheorem{thm}{Theorem}[section] 
\newtheorem{lem}[thm]{Lemma} 
\newtheorem{cor}[thm]{Corollary} 
\newtheorem{prop}[thm]{Proposition} 
 
\newtheorem{defn}{Definition}[section]
\theoremstyle{definition} 
 
\theoremstyle{remark}

\theoremstyle{definition}

\def\O{\Omega}

\def\S{\Sigma} 
\def\n{\nabla}

\def\p{\partial}

\def\a{\alpha}

\def\n{\nabla}

\def\O{\Omega}
\def\p{\partial}

\def\a{\alpha}

\def\d{\delta}

\def\s{\sigma}

\def\ov{\overline}

\def\n{\nabla}
\def\<{\langle}
\def\>{\rangle}

\def\n{\nabla}

\def\NN{\mathbb{N}}

\def\RR{\mathbb{R}}

\def\SS{\mathbb{S}}

\def\O{\Omega}
\def\p{\partial}

\def\a{\alpha}

\def\d{\delta}

\def\K{\mathcal{K}}
\def\s{\sigma}

\def\ov{\overline}

\def\R{\mathbb{R}}

\def\C{\mathcal{C}}

\def\C{\mathcal{C}}
\def\H{\mathcal{H}}
\def\ol{\overline}

{\left\{\begin{array}{@{}l@{}}}{\end{array}\right.}
\patchcmd{\abstract}{\scshape\abstractname}{\textbf{\abstractname}}{}{}
\makeatletter 
\def\@makefnmark{} 
\makeatother

\numberwithin{equation}{section}
\numberwithin{exa}{section}

\makeatletter
\@namedef{subjclassname@2020}{\textup{2020} Mathematics Subject Classification}
\makeatother

\begin{document}
\title{The capillary Christoffel-Minkowski problem}

\author[X. Mei]{Xinqun Mei}
\address[X. Mei]{Key Laboratory of Pure and Applied Mathematics, School of Mathematical Sciences, Peking University,  Beijing, 100871, P.R.China}

\email{\href{qunmath@pku.edu.cn}{qunmath@pku.edu.cn}}

\author[G. Wang]{Guofang Wang}
\address[G. Wang]{Mathematisches Institut, Albert-Ludwigs-Universit\"{a}t Freiburg, Freiburg im Breisgau, 79104, Germany}
\email{\href{guofang.wang@math.uni-freiburg.de}{guofang.wang@math.uni-freiburg.de}}

\author[L. Weng]{Liangjun Weng}
\address[L. Weng]{Centro di Ricerca Matematica Ennio De Giorgi, Scuola Normale Superiore, Pisa, 56126, Italy \& Dipartimento di Matematica, Universit\`a di Pisa, Pisa, 56127, Italy}
\email{\href{mailto:liangjun.weng@sns.it}  {liangjun.weng@sns.it}}



\begin{abstract} In this article, we introduce a $k$-th capillary area measure for capillary convex bodies in the Euclidean half-space, which serves as a boundary counterpart to the classical concept of area measure (see, e.g., \cite[Chapter 8]{Sch}). We then propose a Christoffel-Minkowski problem for capillary convex bodies, to find a capillary convex body in the Euclidean half-space with a prescribed $k$-th capillary area measure. This problem is equivalent to solving a Hessian-type equation with a Robin boundary value condition. We then establish the existence and uniqueness of a smooth solution under a natural sufficient condition.
\end{abstract}

 \maketitle


\section{Introduction}
In convex geometry, the prescribed $k$-th area measure problem is a well-known problem. It asks whether, for a given positive function $f$ on $\SS^{n}$, one can find a closed, strongly convex hypersurface $\Sigma$ with $k$-th symmetric function of the principal curvature radii equal to $f$ as a function on its normals? When $k=n$, it corresponds to the famous Minkowski problem, which has been fully solved through the works of Minkowski \cite{Min}, Alexandrov \cite{Alex}, Lewy \cite{Lewy}, Nirenberg \cite{Nire}, Pogorelov \cite{Pog52, Pog}, Cheng-Yau \cite{CY76} and many others. Analytically, the Minkowski problem is equivalent to solving a Monge-Amp\`ere equation on $\SS^{n}$. The study of the Minkowski problem has played a significant role in advancing research on fully nonlinear equations, especially on the Monge-Amp\`ere equation. We can also see a closely related work by Xia \cite{Xia} for an anisotropic Minkowski problem for closed hypersurfaces.
When $k=1$, it is known as the Christoffel problem, which is equivalent to finding convex solutions to the Laplace equation on $\SS^{n}$.  Firey \cite{Firey, Firey1} and Berg \cite{Berg}  found necessary and sufficient conditions using the Green function of the Laplacian on $\SS^{n}$. More recently, Li-Wan-Wang \cite{LWW} expressed the Christoffel problem as a Laplace equation on the entire space $\RR^{n+1}$ and provided a new and simpler necessary and sufficient condition for the Christoffel problem.  Further details on Christoffel's problem can be found in \cite{Chri,  Hilbert, Hur, Suss}. For the intermediate cases, i.e., $2\leq k\leq n-1$, it is referred to as the Christoffel-Minkowski problem. Guan-Ma \cite{GM}  solved this problem under a continuous homotopy condition, which was removed later by  Sheng-Wang-Trudinger \cite{STW}. In addition to the elliptic method, Chou-Wang \cite{CW00} and Bryan-Ivaki-Scheuer \cite{BIS2023} provided a flow method to solve the above problems.
For a more comprehensive description of the prescribed area measure problem in convex geometry, see \cite{Guan-note} and \cite{Sch}. 

\subsection{Motivation and set up}
In recent years, there has been growing interest in the study of capillary hypersurfaces, including free boundary hypersurfaces, along with their geometry and topology, see, for instance, \cite{FL14, FS16, JWXZ, MWW-GCF, SWX, WW20, WWX22, WX2019, WeX21} and references therein. 
Capillary hypersurfaces originated from physics through the work of Thomas Young. Later,
Young, Liouville, and Gauss formulated it as a variational problem using the so-called wetting energy.  For capillary hypersurfaces, we refer to the elegant book by Finn \cite{Finn}, and also a book by Maggi  \cite{Maggi}. In this paper, with a slight abuse of terminology,
a smooth compact embedded hypersurface in $\ov{{\R}_+^{n+1}}$  whose boundary lies on $\p\ov{{\R}_+^{n+1}}$ is called {\it capillary hypersurface} if
it intersects $\p\ov{{{\R}}_+^{n+1}}$ at a constant contact angle $\theta\in (0,\pi)$. 
   Let $\nu$ be the unit outward normal of $\Sigma$. The contact angle $\theta $ is defined by $$ \cos (\pi -\theta)=\langle \nu, e \rangle,\quad \rm{along~\partial\S},$$
	where $e\coloneqq -E_{n+1}$, $E_{n+1}$ is the $(n+1)$-th unit coordinate vector in $\ov{\RR^{n+1}_+}$ and thus $e$ is the unit outward normal of $\p \ov{\R^{n+1}_+}$.
This means that a capillary hypersurface in this paper is not necessarily a hypersurface of constant mean curvature.
If $\Sigma$ is strictly convex, then it is clear that the image of the Gauss map, namely,  $\nu$, lies on the spherical cap 
\begin{eqnarray*}
		\SS^n_\theta \coloneqq  \left\{ x\in \SS^n \,|\,  \langle x, E_{n+1} \rangle \geq \cos \theta \right\}.
\end{eqnarray*}
Instead of using the usual Gauss map $\nu$, it is more convenient to use the following map
$$\tilde \nu\coloneqq  T \circ \nu: \S \to \C_\theta,$$
where $\C_\theta$ is a spherical cap, defined by
\begin{eqnarray*}
		\C_{\theta }  \coloneqq  \left\{\xi\in \ov{\mathbb{R}^{n+1}_+} \mid |\xi-\cos\theta e|= 1 \right\},
\end{eqnarray*}
	and $T:\SS^n_\theta\to\C_\theta$ is defined by $ T(y)\coloneqq y+\cos\theta e$,  which is a translation in the vertical direction.
 We refer to $\tilde \nu$  as the {\it capillary Gauss map} of $\Sigma$.  It turns out that $\tilde \nu: \S\to \C_\theta$ is a diffeomorphism. Consequently,  one can reparametrize $\S$ using its inverse on $\C_\theta$, see, e.g., \cite[Section 2]{MWWX}. 

In \cite{MWW-AIM}, the authors formulated a capillary Minkowski problem, which concerns the existence of a strictly convex capillary hypersurface $\S\subset \ov{\RR^{n+1}_{+}}$ with a prescribed Gauss-Kronecker curvature on $\C_{\theta}$. We have established the following result in \cite[Theorem 1.1]{MWW-AIM}.

\

\noindent{\bf Theorem A}. 
{\it Let $\theta\in (0,\frac{\pi}{2}]$ and $f\in C^2(\C_\theta)$ be a positive function  satisfying  \eqref{ortho-condition1} below.  Then there exists a $C^{3,\gamma}$ $(\gamma\in(0,1))$ strictly convex capillary hypersurface $\S\subset \ov{\RR^{n+1}_+}$ such that its Gauss-Kronecker curvature $K$ satisfying $$K(\tilde \nu^{-1}(\xi))=
		{f^{-1}(\xi)},$$ for all $\xi\in \C_\theta$. Moreover, $\S$ is unique up to a horizontal translation in $\ov{\RR^{n+1}_+}$.}
  
  \

The capillary Minkowski problem can also be viewed as a prescribed capillary area measure problem, finding a strictly convex capillary hypersurface such that the induced capillary area measure equals a prescribed Borel measure on $\C_{\theta}$. See \cite{MWW-AIM}. The area measures are an important subject in the study of convex bodies, and they are the local versions of quermassintegrals in the Brunn-Minkowski theory.  For more details about the classical Christoffel-Minkowski problem, see \cite[Chapter~8]{Sch}.

The primary objective of this article is to extend our work presented in \cite{MWW-AIM} and study the capillary Christoffel-Minkowski problem or the prescribed  $k$-th capillary area measure problem for the convex body with a capillary boundary in $\ov{\RR^{n+1}_{+}}$.  
 Firstly, we introduce a novel concept of local parallel sets for the capillary convex bodies. We say $\widehat{\S}$ is a \textit{capillary convex body} (a compact convex set with non-empty interior) in $\RR^{n+1}$ if it is a bounded closed region in $\ov{\RR^{n+1}_+}$ enclosed by a strictly convex capillary hypersurface $\S$ and  $\partial \ov{\RR^{n+1}_{+}}$. The class of capillary convex bodies in $\ov{{\mathbb{R}}^{n+1}_+}$ is denoted by $\K_{\theta}$.   Given a Borel set $\beta\subset \C_{\theta}$ and a capillary convex body $\widehat{\S}\subset \K_{\theta}$, we consider the following set :
\begin{eqnarray*}
    B_{s}(\widehat{\S}, \beta)\coloneqq  \left\{Y\in \ov{{\RR}^{n+1}_{+}} \mid  Y=X+t\tilde{\nu}(X),~0<t<s,~\text{for}~X\in \Sigma ~ \text{and} ~\tilde{\nu}(X)\in \beta \right\}.
\end{eqnarray*}
When $\beta=\C_{\theta}$, the set
\begin{eqnarray*}
    A_{s}(\widehat{\Sigma}, \C_{\theta}) \coloneqq \left\{Y\in \ov{{\RR}^{n+1}_{+}} \mid  Y=X+s\tilde{\nu}(X), ~X\in \Sigma~\text{and}~s>0 \right\}
\end{eqnarray*}
is also a strictly convex capillary hypersurface in $\ov{{\RR}^{n+1}_{+}}$ (see, for instance, \cite[Remark 2.17]{MWWX} or \cite[Proposition~2.3]{JWXZ}).
 
Direct calculation yields the following local Steiner-type formula (see also \cite[Remark~2.17]{MWWX} for a global version),
\begin{eqnarray*}
    \text{Vol}(B_{s}(\widehat{\Sigma}, \beta))&=&\sum\limits_{k=0}^{n}\frac{s^{n+1-k}}{n+1-k}\int_{\Sigma \cap \tilde{\nu}^{-1}(\beta)}(1+\cos\theta \<\nu, e\>)\sigma_{n-k}(\kappa)d\mu\notag\\
    &=&\frac{1}{n+1}\sum\limits_{k=0}^{n}s^{n+1-k}\binom{n+1}{k}S_{k}^c(\widehat{\Sigma}, \beta),
\end{eqnarray*}
where $S_k^c(\widehat{\Sigma},\beta)$ is given by
\begin{eqnarray}\label{measure}
    S_{k}^c(\widehat{\Sigma}, \beta)\coloneqq \binom{n}{k}^{-1}\int_{\beta}(\sin^{2}\theta+\cos\theta\<\xi, e\>)\sigma_{k}(r)d{\H}^{n},\quad \text{for}~0\leq k\leq n.
\end{eqnarray}
Here, $\kappa$ and $r$ are the set of the principal curvatures and principal curvature radii of $\Sigma$ respectively, $\sigma_{k}$ is the $k$-th elementary symmetric polynomial function, $d \mathcal{H}^n$ and $d\mu$ are the $n$-dimensional Hausdorff measure on $\C_\theta$
 and  $\Sigma$ respectively. The superscript $c$ indicates the presence of the \textit{capillary} effect in our notation. Moreover, $S_k^c (\widehat{\S}, \C_{\theta})$ $(0\leq k\leq n)$ coincide with the quermassintegrals of convex capillary hypersurface introduced in  \cite[Eq. (1.8)]{WWX22}, see also \cite[Lemma~2.14]{MWWX}. 

 For each capillary convex body $\widehat{\S}$, by \eqref{measure} we introduce  a \textit{$k$-th capillary area measure} on $\C_{\theta}$, which is given  by \begin{eqnarray}\notag
     dS_k^c \coloneqq  \ell\sigma_{k}(\n^{2}h+h\sigma)d\mathcal{H}^{n}, ~~~0\leq k\leq n,
 \end{eqnarray}  
 where $$\ell(\xi) \coloneqq \sin^{2}\theta+\cos\theta\<\xi, e\>,$$ and $h$ is the support function of $\S$. In particular, when $\theta=\frac \pi 2$, $dS_k^c$ coincides with the classical $k$-th area measure for the convex body, see for instance \cite[Section 5.1]{Sch}. The main goal of this paper is to identify a capillary convex body $\widehat{\S}$, such that the  $k$-th capillary area measure is prescribed on $\C_{\theta}$. We formulate the problem as follows.

 \

\noindent	\textbf{Capillary Christoffel-Minkowski problem:} \textit{Given a positive, $C^{2}$ function $f$ on $\C_{\theta}$, when does  there exist a capillary convex body $\widehat{\Sigma}$ in $\ov{\R^{n+1}_+}$ such that
\begin{eqnarray*}
    dS_k^c=\ell fd\mathcal{H}^{n}~?
\end{eqnarray*}}

\
	
When $k=n$, it reduces to the $n$-th capillary area measure and corresponds to the capillary Minkowski problem, which was solved by the authors in \cite{MWW-AIM}. 
For the intermediate case  $2\leq k\leq n-1$, we refer to this problem as the \textit{capillary Christoffel-Minkowski problem}. From Proposition~\ref{pro-2.6} below, the capillary Christoffel-Minkowski problem is equivalent to finding  convex solutions to the following Hessian type equation with Robin boundary value condition,  
\begin{eqnarray}
\label{eq-k}
	\begin{array}{rcll}\vspace{2mm}
		\sigma_{k}(\n^2 h+h\sigma)&=&f,& \quad  \hbox{ in }  \C_{\theta},\\
		\n_\mu h&=&\cot\theta h, &\quad \hbox{ on }  \p \C_\theta,
	\end{array}
 \end{eqnarray}
where   $h$ is an unknown function, $\n h$ and $\n^2h$ are the gradient and the Hessian of $h$ on $\C_\theta$ with respect to the standard spherical metric $\sigma$ on $\C_{\theta}$ respectively  and $\mu$ is the unit outward normal of $\p\C_\theta\subset \C_{\theta}$.
If Eq. \eqref{eq-k} is solvable, the function $f$  satisfies
\begin{eqnarray}\label{ortho-condition1}
    \int_{\C_{\theta}}f(\xi)\<\xi, E_{\alpha}\>d\mathcal{H}^{n}=0,\quad \forall~ 1\leq \alpha\leq n,
\end{eqnarray} (cf. Proposition \ref{pro ortho} for more details)
where  $\{E_i\}_{i=1}^{n+1}$ is the standard coordinate unit vector of $\ov{\RR^{n+1}_+}$. For the capillary Minkowski problem, the necessary condition \eqref{ortho-condition1} is also sufficient, while the situation becomes subtle for the general case $k\neq n$. This phenomenon also occurs in the closed case, see for instance \cite{GM, GMZ, STW}. Before presenting the main result,  we introduce some definitions.
\begin{defn}\label{def1.1}
  We say that a function $f\in C^{2}(\C_{\theta})$ is (strictly) convex if the matrix $\n^{2}f+f\sigma$ is positive semi-definite (respectively, positive definite) on $\C_{\theta}$.
\end{defn}
\begin{defn}
  For $s\in \mathbb{R}$, let $\mathcal{L}_{s}$ be the set of positive $C^{2}$ functions $f$ on $\C_{\theta}$ such that $f$ satisfies \eqref{ortho-condition1} and $f^{s}$ is convex on $\C_{\theta}$, in the sense of Definition \ref{def1.1}. Moreover, we say $f$ is connected to $g$ in $\mathcal{L}_{s}$ if there exists a continuous path $q(t)\in \mathcal{L}_{s}$, $0\leq t\leq 1$, such that $q(0)=g$, $q(1)=f$, and $\n_{\mu}q(t)\geq 0$ on $\partial\C_{\theta}$ for any $t\in[0, 1]$. 
\end{defn}

It is worth noting that, compared to the definition of homotopy given in \cite[Definition 1.1]{GM}, our setting includes an additional requirement: a Neumann boundary condition will be imposed. This condition plays a crucial role in establishing the constant rank theorem for ensuring the convex solutions of Eq. \eqref{eq-k}.

\subsection {Main result}
The main theorem of this paper is as follows.
\begin{thm}\label{main theorem}
  Let $\theta\in (0,\frac{\pi}{2}]$ and $1\leq k\leq n-1$.  Suppose $f\in C^{2}(\C_{\theta})$ is connected to $1$ in $\mathcal{L}_{-\frac{1}{k}}$.  Then there exists a $C^{3, \gamma}$ $(\gamma\in(0,1))$ strictly convex capillary hypersurface $\Sigma$ in $ \ov{{\RR}^{n+1}_{+}}$, such that $\widehat{\S}$ satisfies the prescribed  $k$-th capillary area measure problem. 
  Moreover, $\S$ is unique up to a horizontal translation in $\ov{{\RR}^{n+1}_{+}}$. 
\end{thm}

\

We conclude the introduction by outlining the strategy for solving  Eq. \eqref{eq-k}. If $\theta=\frac{\pi}{2}$, this problem can be reduced to the closed case using a reflection argument. Hence, we focus on the case $\theta<\frac \pi 2$ in proving Theorem \ref{main theorem}. We will employ the method of continuity.  We need to establish a  priori estimates for solutions to Eq. \eqref{eq-k} and show that the convexity is preserved. To establish the $C^{0}$ estimate, we modify the approach in  \cite[Lemma~3]{CY76}, \cite[Lemma 3.1]{MWW-AIM} and combine it with a general form of capillary Alexandrov-Fenchel inequalities for capillary hypersurfaces presented in \cite[Theorem~1.2]{MWWX} for $\theta\in(0, \pi)$.  Hence our $C^0$-estimate for \eqref{eq-k} holds for all contact angle $\theta\in(0, \pi)$, see Lemma \ref{c0 est}. The $C^1$ estimate follows from the strict convexity of $h$ and the $C^0$ estimate. In contrast to the closed case considered in \cite[Proposition 3.2]{GM}, our global $C^2$ estimate depends on the gradient estimates, due to the specific choice of test functions used to derive the boundary $C^2$ estimate. The choice of the test functions is motivated by the works of \cite{LTU, MQ} and \cite{MWW-AIM}.
For the preservation of the convexity, we apply the celebrated result of Guan-Ma \cite{GM} on the Constant Rank Theorem.
It is worth noting that the assumption $\theta\in (0, \frac{\pi}{2})$ is crucial not only for deriving the $C^2$-estimate, but also for showing the preservation of the convexity.
A similar assumption was employed in our previous works \cite{MWW-Weingarten} and \cite{MWW-Lp}, concerning the problems of convex capillary surfaces with prescribed $k$-th Weingarten curvature and prescribed capillary $L_p$-surface area measures for convex capillary bodies, respectively.  Nevertheless, we expect that the above results still hold for $\theta > \frac{\pi}{2}$.

\

\textbf{The rest of the article is structured as follows.} 
	In Section \ref{sec2}, we list some properties of elementary symmetric functions and introduce some notations and basic facts about capillary hypersurfaces. Moreover, we prove the necessary condition \eqref{ortho-condition1} for the capillary Christoffel-Minkowski problem. 
	In Section \ref{sec3}, we establish a priori estimates for the solution to Eq. \eqref{eq-k}. In Section \ref{sec-convex}, under certain assumptions on $f$, we investigate the convexity of the solution to Eq. \eqref{eq-k}.
 In the final section, by using the method of continuity, we prove Theorem \ref{main theorem}.
	\vspace{.2cm}

\section{Preliminaries}\label{sec2}
\subsection{Basic properties of elementary symmetric functions.}
In this subsection, we recall the definitions and some basic properties
of elementary symmetric functions, which can be found in
\cite{L96}.

\begin{defn}
For any $k = 1, 2,\ldots, n,$ we set
\begin{eqnarray*} 
\sigma_k(\lambda) = \sum _{1 \le i_1 < i_2 <\cdots<i_k\leq n}\lambda_{i_1}\lambda_{i_2}\cdots\lambda_{i_k},
 \qquad \text {for any} \quad\lambda=(\lambda_1,\ldots,\lambda_n)\in \mathbb{R}^{n}.
\end{eqnarray*}
For convenience, let $\sigma_0=1$ and $\sigma_k =0$ for $k>n$. 
\end{defn}
Denote by $\sigma _k (\lambda \left| i \right.)$ the symmetric
function with $\lambda_i = 0$ and $\sigma _k (\lambda \left| ij
\right.)$ the symmetric function with $\lambda_i =\lambda_j = 0$. 
Recall that the classic G{\aa}rding's cone is defined as
\begin{eqnarray*}
\Gamma_k  =  \left\{ \lambda  \in \mathbb{R}^n \mid \sigma _i (\lambda ) >
0,\forall 1 \le i \le k \right\}.
\end{eqnarray*}
The following properties of elementary symmetric functions are well known.
\begin{lem}\label{prop2.2}
Let $\lambda=(\lambda_1,\ldots,\lambda_n)\in\mathbb{R}^n$ and $k
=1, \ldots, n$. Then
\begin{enumerate}
    \item $\sigma_k(\lambda)=\sigma_k(\lambda|i)+\lambda_i\sigma_{k-1}(\lambda|i), \quad \forall \,1\leq i\leq n$.
    \item $\sum\limits_{i=1}^n \lambda_i\sigma_{k-1}(\lambda|i)=k\sigma_{k}(\lambda)$.
    \item $\sum\limits_{i=1}^n \sigma_{k}(\lambda|i)=(n-k)\sigma_{k}(\lambda)$.
    \item If $\lambda \in \Gamma_k$ and $\lambda_1 \geq  \lambda_2 \geq \cdots \geq \lambda_n$, then
\begin{eqnarray*} 
\sigma_{k-1} (\lambda|n) \geq\sigma_{k-1} (\lambda|n-1) \geq \cdots \geq \sigma_{k-1} (\lambda|1) >0.
\end{eqnarray*}
\end{enumerate}

\end{lem}
\begin{defn}
    Let $A=\{A_{ij}\}$ be an $n\times n$ symmetric matrix, for $k=1,2,\cdots, n$, we define
    \begin{eqnarray}\label{k-ele}
        \sigma_{k}(A)\coloneqq \sigma_{k}(\lambda(A))=\sum\limits_{1\leq i_{1}<i_{2}\cdots< i_{k}\leq n}\lambda_{i_{1}}(A)\lambda_{i_{2}}(A)\cdots \lambda_{i_{k}}(A),
    \end{eqnarray}
    where $\lambda(A)=(\lambda_{1}(A), \lambda_{2}(A), \cdots, \lambda_{n}(A))$ are the eigenvalues of the symmetric matrix $A$. Alternatively, \eqref{k-ele} is also the sum of its $k\times k$ principal minors.
\end{defn}
We also denote by $\sigma _m (A \left|
i \right.)$ the symmetric function with $A$ deleting the $i$-row and
$i$-column and $\sigma _m (A \left| ij \right.)$ the symmetric
function with $A$ deleting the $i,j$-rows and $i,j$-columns. Then
we have the following identities.
\begin{lem}
Suppose that  $A=\{A_{ij}\}$ is diagonal, and $1\leq k\leq n$,
then
\begin{eqnarray*}
\sigma_{k-1}^{ij}(A)= \begin{cases}
\sigma _{k- 1} (A\left| i \right.), &\text{if } i = j \\
0, &\text{if } i \ne j
\end{cases},
\end{eqnarray*}
where $\sigma_{k-1}^{ij}(A)=\frac{{\partial \sigma _k (A)}} {{\partial A_{ij} }}$.
\end{lem}
\begin{lem}\label{prop2.3}
For $\lambda \in \Gamma_k$ and $k > l \geq 0$, $ r > s \geq 0$, $k \geq r$, $l \geq s$, the generalized Newton-Maclaurin inequality is 
\begin{eqnarray*}
\left(\frac{{\sigma _k (\lambda )}/{C_n^k }}{{\sigma _l (\lambda )}/{C_n^l }}\right)^{\frac{1}{k-l}}
\le \left(\frac{{\sigma _r (\lambda )}/{C_n^r }}{{\sigma _s (\lambda )}/{C_n^s }}\right)^{\frac{1}{r-s}}.
\end{eqnarray*}
Equality holds if and only if $\lambda_1 = \lambda_2 = \cdots =\lambda_n >0$. 
\end{lem}
For proofs of the above lemmas, see, e.g., \cite[Theorem 1.4, Lemma 2.10]{Spruck}.

\subsection{Some properties of strictly convex capillary hypersurfaces in half-space}
Let $\S\subset \ov{{\RR}^{n+1}_+}$ be  a smooth, properly embedded  capillary hypersurface in $\ov{{\RR}^{n+1}_+}$ with boundary $\p \Sigma \subset \p \ov{\R^{n+1}_+}$.
	Let $\widehat{\S}$ denote the domain enclosed by $\Sigma$ and  the hyperplane $\p \ov{{\R}^{n+1}_+}$, and  let $\widehat{\p \Sigma}$ denote the domain in  $\p \ov{{\R}^{n+1}_+}$ enclosed  by $\p \S$. 
Let $\mu$ be the unit outward co-normal of $\p\S$ in $\S$ and  $\overline{\nu}$ the unit outward normal to $\partial\Sigma$ in $\partial\ov{\mathbb{R}^{n+1}_+}$ 
	such that $\{\nu, \mu\}$ and $\{\overline{\nu},e\}$ have the same orientation in the  normal bundle of $\partial\Sigma\subset\ov{{\R}^{n+1}_+}$. 
	It  follows  
	\begin{eqnarray}\label{co-normal bundle}
		\begin{array} {rcl}
			e &=&\sin\theta \mu-\cos\theta \nu,
			\\
			\overline{\nu} &=&\cos\theta \mu+\sin\theta \nu.
		\end{array}
	\end{eqnarray}
Assume further that $\S\subset \ov{ \RR^{n+1}_+}$ is strictly convex, we parametrize the capillary hypersurface $\S$ using  the inverse capillary Gauss map (refer to \cite[Section~2]{MWWX}), which is given by the map $X:\C_{\theta}\to \S$
	\begin{eqnarray*}
		X(\xi)\coloneqq \tilde \nu ^{-1} (\xi) =\nu^{-1}\circ T^{-1}(\xi)=\nu^{-1}(\xi-\cos\theta e).
	\end{eqnarray*}
	The usual support function of $\Sigma$ is  defined by 
	\begin{eqnarray*}
		h(X)\coloneqq \<X,\nu(X)\>.
	\end{eqnarray*}
	The support function $h$ is equivalently defined  by the unique decomposition
	$$X= h(X) \nu(X)+ t,  \qquad t\in T_X\Sigma.$$ 
 Without causing confusion,  we also view the support function defined on $\C_{\theta }$ by our parametrization, namely
	\begin{eqnarray*}
		h(\xi)\coloneqq \<X(\xi),\nu(X(\xi ))\>=\<X(\xi), T^{-1} (\xi )\>=
		\< \tilde \nu ^{-1}(\xi)  , \xi -\cos\theta e\>.
	\end{eqnarray*}
It is clear that  the capillary Gauss map for $ \C_\theta$  is the identity map 
from $\C_\theta \to \C_\theta$ and the support function of $\C_{\theta}$ is
\begin{eqnarray*}
		\ell(\xi)=\<\xi  , \xi -\cos\theta e\>=
		\sin ^2\theta + \cos\theta\< \xi, e\>.
	\end{eqnarray*} Moreover, it is easy to see that $\ell$ satisfies Eq. \eqref{eq-k} with $f\equiv \binom{n}{k}$. 
The following lemma was established in \cite[Lemma~2.4]{MWWX}, and we include it here for later use.
 \begin{lem}\label{parametrization}
 For the parametrization  $X:\C_{\theta}\to \S$, there hold
	\begin{itemize}
	\item[(1)] $X(\xi)=\nabla h(\xi) +h(\xi)T^{-1}(\xi)$ for $\xi\in \C_\theta$.
    
		 \item[(2)]  $\n_\mu h=\cot \theta h$ along  $\p \C_{\theta}$, where $\mu$ is the outward unit co-normal to $\p \C_{\theta}\subset \C_{\theta}$.
         
	 \item[(3)] The principal radii of $\S$ at $X(\xi)$ are given by the eigenvalues of $$W\coloneqq (\n^2 h+ h\s)$$  with respect to the metric $\s$. In particular, $W>0$ on $\C_{\theta}$.
	 \end{itemize}	
     \end{lem}

From \cite[Proposition~2.6]{MWWX}, there is a one-to-one correspondence between a strictly convex capillary hypersurface and a strictly convex function $h\in C^{2}(\C_{\theta})$ satisfying a Robin boundary condition as in Lemma \ref{parametrization} (2). It then follows readily that
	\begin{prop}\label{pro-2.6}
 The capillary  Christoffel-Minkowski problem 
		is equivalent to finding a strictly convex solution to  Eq.  \eqref{eq-k}.\end{prop}

At the end of this section, we prove a necessary condition \eqref{ortho-condition1} for the solvability of the capillary Christoffel-Minkowski problem.
	\begin{prop}\label{pro ortho}
		If  Eq. \eqref{eq-k} is solvable, then $f$ satisfies
  
		\begin{eqnarray}\label{ortho-condition}
    \int_{\C_{\theta}}f(\xi)\<\xi, E_{\alpha}\>d\H^{n}=0,\quad \forall~ 1\leq \alpha\leq n.
\end{eqnarray}		
	\end{prop}
	\begin{proof}
    From \cite[Lemma 18.30]{Ben}, we know that $\sum\limits_{i=1}^{n}\left(\sigma_{k-1}^{ij}(W)\right)_{i}=0$, together with Lemma \ref{prop2.2}, and using integrating by parts twice, we have
 \begin{eqnarray*}
     \int_{\C_{\theta}}f(\xi)\<\xi, E_{\alpha}\>d\H^{n}
     &=&\frac{1}{k}\int_{\C_{\theta}}\<\xi, E_{\alpha}\>\sigma_{k-1}^{ij}(W)W_{ij}d\H^{n}\\
     &=&\frac{1}{k}\int_{\C_{\theta}}h \sigma^{ij}_{k-1}(W)\left(\<\xi, E_{\alpha}\>_{ij}+\<\xi, E_{\alpha}\>\delta_{ij}\right)d\H^{n}\\
     &&+\frac{1}{k}\int_{\partial \C_{\theta}}\sigma_{k-1}^{ij}(W) \left(\<\xi, E_{\alpha}\>h_{i}\mu^{j}- h(\<\xi, E_{\alpha}\>)_{j}\mu^{i}\right)d\H^{n-1}.
 \end{eqnarray*}
 Let $\{e_i\}_{i=1}^n$ be an orthonormal frame on $\C_\theta$ such that $e_n=\mu$ on $\p \C_\theta$.
 On $\partial\C_{\theta}$,  from \eqref{co-normal bundle}, we have $\<\xi, E_{\alpha}\>_{\mu}=\cot\theta \<\xi, E_{\alpha}\>$, and by \cite[Proposition~2.8]{MWW-AIM}, we have $h_{j n}=0$ for $1\leq j\leq n-1$, which implies $W_{j n}=W(e_j, e_n)=0$ for $1\leq j \leq n-1$. 
Altogether yield 
 \begin{eqnarray*}
    && \int_{\partial\C_{\theta}}\sigma_{k-1}^{ij}\left(\<\xi, E_{\alpha}\>h_{i}\mu^{j}- h(\<\xi, E_{\alpha}\>)_{j}\mu^{i}\right)d\H^{n-1}\\
    &=&\int_{\partial \C_{\theta}}\sigma_{k-1}^{nn}\left(\<\xi, E_{\alpha}\>h_{\mu}-h\<\xi, E_{\alpha}\>_{\mu}\right)d\H^{n-1}=0.
 \end{eqnarray*}
Together with the simple fact that $\<\xi, E_{\alpha}\>_{ij}+\<\xi, E_{\alpha}\>\delta_{ij}=0$ for $1\leq \a\leq n$, on $\C_{\theta}$, then the desired conclusion \eqref{ortho-condition} follows.
 \end{proof}

\section{A priori estimates}\label{sec3}

In this section, we establish the a priori estimates for the strictly convex solution of Eq. \eqref{eq-k}.  
For any solution $h$ to Eq. \eqref{eq-k}, it is easy to verify that  
\begin{eqnarray*}
	h(\xi)+\<\xi,E_\alpha\> ,\quad \text{ for all } 1\leq \a\leq n,
\end{eqnarray*} is still a solution. Up to a horizontal translation, we assume $h$ satisfies the following orthogonal condition:
\begin{eqnarray}\label{orth cond}
	\int_{\C_\theta} h(\xi) \<\xi,E_\alpha\> d\H^{n}=0, \quad \text{ for all } 1\leq \a\leq n.
\end{eqnarray}
Under the condition \eqref{orth cond},  we are able to establish the following a priori estimates.
\begin{thm}\label{priori estimate}
Let $\theta\in (0, \frac{\pi}{2})$. Suppose that $h$ is a strictly convex solution of  Eq. \eqref{eq-k} and satisfies the condition \eqref{orth cond}. For any integer $m\geq 2$ and $\gamma\in(0,1)$, there exists a constant $C>0$  depending on  $n, k,$ $\min\limits_{\C_{\theta}} f$ and $\|f\|_{C^{m}(\C_\theta)}$, such that
	\begin{eqnarray}\label{sch est}
		\|h\|_{C^{m+1,\gamma}(\C_\theta)}\leq C.
	\end{eqnarray}
\end{thm}

Before proving Theorem \ref{priori estimate}, certain preparations are necessary. In the following,  for convenience,  when dealing with the tensors and their covariant derivatives on spherical cap $\C_\theta$, we will use a local frame to express tensors with the help of their components, and indices appearing after the semi-colon denote the covariant derivatives. For example, let $\{e_i\}_{i=1}^n$ be an orthonormal frame on $\C_\theta$,   the expression $h_{ij}$ denotes $\n^2 h(e_i,e_j)$ and   $W_{ijk}\coloneqq \n_{e_k}(W_{ij})$ etc.

Denote
\begin{eqnarray}\label{eq3}
	 F(W)=\sigma_{k}^{\frac{1}{k}}(W)=f^{\frac{1}{k}}\coloneqq \tilde f,
\end{eqnarray}  and
	\begin{eqnarray*}
		F^{ij}\coloneqq \frac{\partial F}{\partial W_{ij}},
		\quad  F^{ij,kl}\coloneqq \frac{\partial^{2}F}{\partial W_{ij}\partial W_{kl}}.
	\end{eqnarray*}
We use the Einstein summation convention that the repeated indices are implicitly summed over, no matter whether the indices are upper or lower. In case of possible confusion, we will indicate the summation.

\subsection{\texorpdfstring{$C^0$}{} estimate}\ 
In this subsection, we follow the idea of Cheng-Yau in \cite[Lemma~3]{CY76}, Guan-Ma \cite[Lemma 3.1]{GM}, and also our previous work \cite[Lemma 3.1]{MWW-AIM}. By adapting a general form of the Alexandrov-Fenchel inequalities for convex capillary hypersurfaces in $\ol{\RR^{n+1}_+}$ with $\theta\in(0,\pi)$ (cf. \cite[Theorem 1.2] {MWWX}), we get the uniform bound for strictly convex solutions of Eq. \eqref{eq-k} for all $\theta\in(0,\pi)$.

\begin{lem}\label{c0 est}
Let $\theta\in(0, \pi)$ and $h$ be a strictly convex solution of  Eq. \eqref{eq-k} that satisfies \eqref{orth cond}.  
Then there exists a positive constant $C$ depending  on $n, \min\limits_{\C_{\theta}}f$, and $\max\limits_{\C_\theta} f$, such that
\begin{eqnarray*}
	0<\min_{\C_\theta} h\leq \max_{\C_\theta} h\leq C.
\end{eqnarray*} 
\end{lem}
\begin{proof}  
First of all, due to \eqref{orth cond}, by \cite[Lemma 3.1]{MWW-AIM}, we have $h>0$ and the origin lies in the interior of $\widehat{\partial\Sigma}$.	
Let $R_{0}>0$ be the smallest number of $R$ such that 
	\begin{eqnarray}\label{enclosed}
	\widehat	\S\subset \widehat{R\C_\theta},
	\end{eqnarray}where   $R\C_\theta\coloneqq \{\xi \in \ov{\RR^{n+1}_+}: |\xi-\cos\theta R e|=R\}$. It is clear that there exists $X_0\in \Sigma \cap R_0\C_{\theta } $.  Denote $\widehat{X}_0\coloneqq \frac{X_0 }{R_0}\in \C_\theta$.
For  any $\xi \in \C_\theta$, we have 
	\begin{eqnarray}\label{lower bdd of h}
	h(\xi)&=&\max_{Y\in \S}\<\xi-\cos\theta e, Y\>\notag \\
    &\geq& \max \left\{0, \<\xi-\cos\theta e,X_0\>  \right\}\notag \\
    &=&  \max \left\{0, R_0\<\xi-\cos\theta e,\widehat X_0\> \right\}.
	~~~\end{eqnarray} 
Multiplying $f$ on both sides of \eqref{lower bdd of h} and integrating over $  \C_\theta$, we see
\begin{eqnarray}\label{R1}
  R_0\leq \left(\int_{\C_\theta}h(\xi)\cdot  f(\xi)d \H^{n}\right) \left(\int_{\C_\theta}\max \left\{0,  \<\xi-\cos\theta e,\widehat X_0\> \right \}\cdot f(\xi) d  \H^{n}\right) ^{-1}.
\end{eqnarray}
Note that 
\begin{eqnarray}\label{R2}
    \int_{\C_{\theta}}h(\xi)\cdot f(\xi) d \H^{n}=
\int_{\tilde{\nu}(\Sigma)}h(\xi)\cdot f(\xi)d\H^{n}=\int_{\S}\<X, \nu\>\sigma_{n-k}(\kappa)d\mu.
\end{eqnarray}
Together with the  Alexandrov-Fenchel inequalities  for convex capillary hypersurfaces \cite[Theorem~1.2]{MWWX} for $\theta\in(0, \pi)$ and the  Minkowski formulas for capillary hypersurfaces \cite[Proposition~2.6]{WWX22}, we obtain
\begin{eqnarray}
    C_{n, k}\left(\int_{\S}\<X, \nu\>\sigma_{n-k}(\kappa)d\mu\right)^{\frac{1}{k+1}}&\leq& \left(\int_{\Sigma}\<X, \nu\>\sigma_{n-k+1}(\kappa)d\mu\right)^{\frac{1}{k}}\notag \\
    &=&\left(\frac{k}{n-k+1}\int_{\C_{\theta}}f(\xi)d\H^{n}\right)^{\frac{1}{k}}.\label{R3}
\end{eqnarray}
Together with \eqref{R1}, \eqref{R2} and \eqref{R3}, we conclude that 
\begin{eqnarray}
    R_{0}
    &\leq& C_{n,k}\left(\int_{\C_{\theta}}f(\xi)d\H^{n}\right)^{\frac{k+1}{k}}\cdot\left(\inf\limits_{\eta\in \C_{\theta}}\int_{\C_\theta}\max\{0,  \<\xi-\cos\theta e, \eta\>\}\cdot f(\xi) d \H^{n}\right) ^{-1}.\quad ~~~\label{upper bound of R0}
\end{eqnarray}
Using \eqref{enclosed} and \eqref{upper bound of R0}, we see  	
\begin{eqnarray*}
	h(\xi) =\<\xi-\cos\theta e , \tilde{\nu}^{-1}(\xi)\>\leq |\tilde{\nu}^{-1}(\xi)| <2R_0\leq C(n, \min_{\C_\theta} f,\max_{\C_\theta} f).
\end{eqnarray*}
 We complete the proof of Lemma \ref{c0 est}.
\end{proof}
\subsection{\texorpdfstring{$C^1$ }{}estimate}
 In this subsection, we get the gradient estimate for the strictly convex solution $h$ of Eq. \eqref{eq-k}. This estimate is a direct consequence of the convexity and the $C^{0}$ estimate.
\begin{lem}\label{C1 est}
    Let $\theta\in (0, \pi)$, and let $h$  be a strictly convex solution of Eq.  \eqref{eq-k}. Then 
     \begin{eqnarray}\label{gradient ps}
       \max\limits_{\C_{\theta}} |\n h|\leq C,
    \end{eqnarray}
    where the constant $C$ only depends on  $\|h\|_{C^{0}(\C_{\theta})}$.
\end{lem}
\begin{proof}
    Consider the  function
    \begin{eqnarray*}
        P=|\n h|^{2}+h^{2}.
    \end{eqnarray*}
Suppose that $P$ attains its maximum value at some point, say $\xi_{0}\in \C_{\theta}$. If $\xi_{0}\in \C_{\theta}\setminus \partial\C_{\theta}$, by the maximal condition, we have
\begin{eqnarray*}
    0=\n_{e_i}P=2h_{k}h_{ki}+2hh_{i},\quad \text{for}~1\leq i\leq n.
\end{eqnarray*}
Since $h$ is a strictly convex function, we have $\n h(\xi_{0})=0$.  Combining with Lemma \ref{c0 est},  we conclude that \eqref{gradient ps} holds.

If $\xi_{0}\in \partial\C_{\theta}$, from \cite[Proposition 2.8]{MWW-AIM}, 
\begin{eqnarray}\label{h1n}
    h_{\alpha n}=0, ~~\rm{for~any}~~ 1\leq \alpha\leq n-1.
\end{eqnarray}
Using the maximal condition again,  we have 
\begin{eqnarray*}
   0= \n_{e_{\alpha}} P=2\sum\limits_{i=1}^{n}h_{i}h_{i\alpha}+2hh_{\alpha},
\end{eqnarray*}
which implies 
\begin{eqnarray}\label{tangent}
    (\n_{e_{\alpha}}h)(\xi_{0})=0.
\end{eqnarray}
Together with  \eqref{tangent} and Lemma \ref{c0 est}, we get
\begin{eqnarray*}
    |\n h|^{2}\leq (|\n h|^{2}+|h|^{2})(\xi_{0})=(|h_{\mu}|^{2}+h^{2})(\xi_{0})=(1+\cot^{2}\theta)h^{2}(\xi_{0})\leq C.
\end{eqnarray*}
This completes the proof.

\end{proof}

\subsection{\texorpdfstring{$C^2$}{} estimate}\
In this subsection, we establish the $C^2$ estimate.  
This approach was initiated by Lions-Trudinger-Urbas  \cite{LTU} for studying the Monge-Amp\`ere equation and later developed by Ma-Qiu \cite{MQ} for the $k$-Hessian equation with the Neumann boundary condition on a uniformly convex domain in Euclidean space. It is notable to remark that in comparison with the $C^{2}$ estimate of admissible solutions to the $k$-Hessian equation in \cite{MQ}, the $C^{2}$ estimate is relatively simple in our setting once we establish the convexity of the solution.

 We divide the proof into two steps. First, we reduce the global $C^2$ estimate to the boundary double normal $C^2$ estimate, that is, Lemma \ref{c2 to normal}. Second, we derive the boundary double normal $C^2$ estimate by constructing a suitable test function, as shown in Lemma \ref{c2 est}. In particular, the convexity of $\p \C_\theta\subset \C_\theta$ plays an important role in deriving a priori $C^2$ estimate, that is, we  assume that $\theta\in(0,\frac{\pi}{2})$.  More precisely,  this condition is used in both of the aforementioned steps: one is in \eqref{1n} and another is in \eqref{hessian of zeta}.

\begin{lem}\label{c2 to normal}
Let $\theta\in (0, \frac{\pi}{2})$. Suppose that $h$ is a strictly convex solution of  Eq. \eqref{eq-k} that satisfies \eqref{orth cond}. Denote $M\coloneqq \sup\limits_{\partial \C_{\theta}}|\n^{2}h(\mu, \mu)|$, 
	then
	\begin{eqnarray}\label{nor-1}
		\sup\limits_{\C_{\theta}}|\n^{2}h|\leq C(1+M),
	\end{eqnarray}
	where the positive constant $C$ depends on $n, \|f\|_{C^2}$, and $\|h\|_{C^{1}}$.
\end{lem}
\begin{proof}
	We consider the auxiliary function 
	\begin{eqnarray*}
		P(\Xi, \xi)=\n^{2}h(\Xi, \Xi)+h(\xi), 
	\end{eqnarray*}
	for $\xi\in\C_{\theta}$ and $\Xi\in T_{\xi}\C_{\theta}$ is a unit vector.

	Suppose that $P$ attains its maximum at some   point $\xi_{0} \in \C_{\theta} $ along some unit vector $\Xi_{0}\in T_{\xi_{0}}\C_{\theta}$.
	Next, we split the proof into two cases.

   \
	
	{\bf Case  1.} $\xi_{0} \in \C_\theta \setminus\p \C_\theta$. In this case, we choose an orthonormal frame  $\{e_{i}\}_{i=1}^{n}$ around $\xi_{0}$, such that 
	$(W_{ij})=(h_{ij}+h\d_{ij})$ is diagonal at $\xi_{0}$ and $\Xi_0=e_1$.

	Direct calculation yields
	\begin{eqnarray*} 
		P_{i}=h_{11 i}+h_{i},
	\end{eqnarray*}
	and 
	\begin{eqnarray*} 
		\begin{aligned}
 	 F^{ij}P_{ij}= F^{ij}(h_{11 ij}+h_{ij} ).
		\end{aligned}
	\end{eqnarray*}
Note that 
	\begin{eqnarray}\label{sum-1}
		F^{ii}h_{ii}=F^{ii}(W_{ii}-h)=\widetilde f-h \sum\limits_{i=1}^{n}F^{ii}
		\geq -C\left(1+\sum\limits_{i=1}^{n}F^{ii}\right),
	\end{eqnarray}
	where the positive constant $C$ depends on $\|f\|_{C^{0}}$ and $\|h\|_{C^{0}}$. 
Since $W$ is diagonal at $ \xi_{0}$, it follows that $F^{ij}$ is also diagonal at $\xi_{0}$, and
	\begin{eqnarray}\label{sum-2}
		\sum\limits_{i=1}^{n}\sum\limits_{k\neq i}F^{ik}h_{ki}=0.
	\end{eqnarray}
 
Differentiating Eq. \eqref{eq3} once and twice,  we obtain
	\begin{eqnarray}\label{one der}
		F^{ij}W_{ij1}=\widetilde f_{1},
	\end{eqnarray}
	and 
	\begin{eqnarray}\label{tw0 der}
		F^{ij}W_{ij 11}=\widetilde f_{11}-F^{ij, kl}W_{ij1}W_{kl1}\geq \widetilde f _{11},
	\end{eqnarray}
 where the last inequality used the concavity of $\sigma^{\frac{1}{k}}_{k}(W)$.
 
From \eqref{one der} and  \eqref{tw0 der}, it follows	\begin{eqnarray}\label{sum one deri}		F^{ij}h_{ij1}=F^{ij}W_{ij1}-h_{1}\sum\limits_{i=1}^{n}F^{ii}		\geq -C \left(1+\sum\limits_{i=1}^{n}F^{ii} \right),	\end{eqnarray}
 and
 	\begin{eqnarray}\label{sum two deri}		F^{ij}h_{ij11}=F^{ij}W_{ij11}-h_{11}\sum\limits_{i=1}^{n}F^{ii}\geq -C-h_{11}\sum\limits_{i=1}^{n}F^{ii}.	\end{eqnarray}
For the standard metric on spherical cap $\C_\theta$,  we have the commutator formulae 
\begin{eqnarray}\label{forth comm}
	h_{klij}=h_{ijkl}+2h_{kl}\d_{ij}-2h_{ij}\d_{kl}+h_{li}\d_{kj}-h_{kj}\d_{il}.
\end{eqnarray}
Substituting \eqref{forth comm} into \eqref{sum two deri} and  combining \eqref{sum-1}, \eqref{sum-2} and  Lemma \ref{C1 est}, we have 
	\begin{eqnarray}\label{sum-4}
		\begin{aligned}
			F^{ij}h_{11 ij}=&F^{ij}\left(h_{ij11}+2h_{11}\delta_{ij} -2 h_{ij}+h_{1i}\d_{1j} -h_{1j}\d_{i1}\right)
			\\
			\geq &F^{ij}h_{ij11}+2h_{11}\sum\limits_{i=1}^{n}F^{ii}-C \left(1+\sum\limits_{i=1}^{n}F^{ii} \right)\\
			\geq &h_{11}\sum\limits_{i=1}^{n}F^{ii}-C \left(1+\sum\limits_{i=1}^{n}F^{ii} \right),
		\end{aligned}
	\end{eqnarray}where the positive constant $C$ may vary from line to line and  only depends  on $n$, $\|h\|_{C^1}$ and $\|f\|_{C^2}$.

By the maximal condition and \eqref{sum-4}, we derive	\begin{eqnarray}
0&\geq &	 F^{ij}P_{ij}=
	 F^{ij}h_{11ij}+F^{ij}h_{ij}
	\notag  \\&\geq &
	 h_{11} \sum\limits_{i=1}^{n}F^{ii} -C \left(1+\sum\limits_{i=1}^{n}F^{ii}\right).\label{key-ine}
	\end{eqnarray}
From Lemma \ref{prop2.3},
\begin{eqnarray}\label{lower bound}
	\sum\limits_{i=1}^{n}F^{ii}=\frac{n-k+1}{k}\sigma_{k}^{\frac{1}{k}-1}\sigma_{k-1}\geq \binom{n}{k}^{\frac{1}{k}}.
\end{eqnarray} 
 Together with  \eqref{key-ine} and  \eqref{lower bound}, we obtain 
\begin{eqnarray}\label{11}
	h_{11}( \xi_{0})\leq C\left(1+\binom{n}{k}^{-\frac{1}{k}}\right).
\end{eqnarray}

\
{\bf Case 2.}  $ \xi_{0} \in \partial \C_{\theta}$. In this case, we can follow the same argument as in \cite[Proof of Lemma 3.3, Case 2]{MWW-AIM} to obtain 
		\begin{eqnarray}\label{1n}
			\n^2	h(\Xi_0,\Xi_0)  \leq  |h_{\mu\mu}|(\xi_{0})+2 \|h\|_{C^0(\C_{\theta})}.
		\end{eqnarray}

	Finally, combining \eqref{11} and \eqref{1n}, we conclude that \eqref{nor-1} holds.

\end{proof}

Next, we establish the double normal derivative estimate of $h$ on the boundary $\partial\C_{\theta}$. Together with Lemma \ref{c2 est}, this yields the a priori $C^2$ estimate.
In order to control the double normal derivative estimate of $h$, we will construct some barrier functions of $h_\mu$ near the boundary. The auxiliary functions here are inspired by the works in \cite{Gb99, LTU, MQ} and \cite{MWW-AIM}.

To begin with, we introduce an important function 
\begin{eqnarray*}
	\zeta(\xi)\coloneqq e^{- d( \xi)}-1,
\end{eqnarray*}
where $d(\xi)=dist( \xi, \partial \C_{\theta})$. The function
$\zeta(\xi)$ has been used in \cite[Lemma~3.1]{Gb99} and \cite[Lemma~3.4]{MWW-AIM}. It is easy to see that $\zeta|_{\partial \C_{\theta}}=0$, and $\n \zeta|_{\partial \C_{\theta}}=\mu$. Meanwhile, near $\partial \C_{\theta}$, there exists a small positive constant $\delta_{0} (0<\delta_{0}<1)$ such that
\begin{eqnarray}\label{hessian of zeta}
(	\n^{2}_{ij}\zeta) \geq \min \left\{\cot\theta,1 \right\}  \s,\quad {\rm{in}}\quad \Omega_{\delta_{0}},
\end{eqnarray}
where $\Omega_{\delta_{0}}\coloneqq  \left\{ \xi\in \C_{\theta}: d( \xi)\leq \delta_{0} \right\}$.

\begin{lem}\label{c2 est}
Let $\theta\in (0, \frac{\pi}{2})$. Suppose	 $h$ is a strictly convex solution of Eq.  \eqref{eq-k} that satisfies \eqref{orth cond}. 
	Then
	\begin{eqnarray*}
		\sup\limits_{\C_{\theta}}|\n^{2}h|\leq C,
	\end{eqnarray*}
	where the positive constant $C$ depends only on $n, k,  \|f\|_{C^2}$  and $\|h\|_{C^{1}}$.
\end{lem}
\begin{proof}
	We consider an auxiliary function
		\begin{eqnarray*}
			Q(\xi)\coloneqq \<\n h, \n \zeta\>-\left(A+\frac{1}{2}M\right)\zeta(\xi)-\cot\theta h(\xi), \quad  \xi\in \Omega_{\delta_{0}},
		\end{eqnarray*}
		where	$A$ is a positive constant to be determined later.
		
		Assume that $Q $ attains its minimum value at $\xi_{0}\in \left(\Omega_{\delta_{0}}\setminus \partial \Omega_{\delta_{0}}\right)$, and choose an orthonormal frame  $\{e_{i}\}_{i=1}^{n}$ around $\xi_{0}$  such that  $(W_{ij})$ is diagonal at $\xi_{0}$. Recall the commutator formula
        \begin{eqnarray}\label{third comm}
            h_{kij}=h_{ijk}+h_{k}\delta_{ij}-h_{j}\delta_{kij},
        \end{eqnarray}
combining \eqref{third comm},  \eqref{sum one deri}, \eqref{hessian of zeta} and Lemma \ref{C1 est}, at $\xi_0$, we obtain
		\begin{eqnarray*}
			0&\leq & 	 F^{ij}Q_{ij} 
			\\&=&F^{ij}h_{kij}\zeta_{k}+F^{ij}h_{k}\zeta_{kij}+2F^{ij}h_{ki}\zeta_{kj}-\left(A+\frac{1}{2}M\right)F^{ii}\zeta_{ii}-\cot\theta F^{ii}h_{ii}
			\\&=&F^{ii}\left(W_{ii; k}-h_{i}\delta_{ki}
			\right)\zeta_{k}+2F^{ii}(W_{ii}-h)\zeta_{ii}-\left(A+\frac{1}{2}M\right)F^{ii}\zeta_{ii}\\
			&&+F^{ii}h_{k}\zeta_{kii}-\cot\theta F^{ii}h_{ii}.
			\\&\leq  & C\left(1+\sum\limits_{i=1}^{n}F^{ii}\right)-\frac{1}{2}\left(A+\frac{1}{2}M\right)\min\{\cot\theta,1\}\sum\limits_{i=1}^{n}F^{ii},
		\end{eqnarray*}
		where the constant $C$ depends on $n,   \|f\|_{C^1}$ and $\|h\|_{C^{1}}$. 
		In view of \eqref{lower bound},
		it follows 
		\begin{eqnarray*}
			C\left(1+\sum\limits_{i=1}^{n}F^{ii}\right)-\frac{1}{2}\left(A+\frac{1}{2}M\right)\min\{\cot\theta,1\}\sum\limits_{i=1}^{n}F^{ii}<0,
		\end{eqnarray*} 
		if $A$ is chosen 
		 by   
		\begin{eqnarray}\label{chosen of A}
			A\coloneqq \frac{4C\left(1+\binom{n}{k}^{-\frac{1}{k}}\right)}{\min\{\cot\theta,1\}}+\frac{1}{1-e^{-\delta_{0}}}\max\limits_{   \C_{\theta}}(|\n h|+\cot\theta h). 
		\end{eqnarray} 
		   This  contradicts  $F^{ij}Q_{ij}\geq 0$ at $\xi_0$.  Therefore, $\xi_0\in \p \O_{\d_0}$.
		
		If  $\xi_0\in \partial \C_{\theta}\cap \p \O_{\d_0}$, it is easy to see that $Q(\xi_0)=0$. 
		
		If $\xi_0\in \partial \Omega_{\delta_{0}}\setminus  \partial \C_{\theta} $, we have $d(\xi_0)=\delta_0$, and from \eqref{chosen of A},
		\begin{eqnarray*}
			Q(\xi)\geq -|\n h|+A(1-e^{-\delta_{0}})-\cot\theta h\geq 0.
		\end{eqnarray*}
		
		In conclusion, we deduce that 
		\begin{eqnarray*}
			Q(\xi)\geq 0,\quad{\rm{in}}\quad \Omega_{\delta_{0}}.
		\end{eqnarray*} 
		
		\
		
		Now we are ready to obtain the double normal second derivative estimate of $h$. 
		Assume  $h_{\mu\mu}(\eta_{0})\coloneqq \sup\limits_{\partial \C_{\theta}}h_{\mu\mu}>0$ for some $\eta_0\in \p \C_\theta$. In view of \eqref{h1n}, Lemma \ref{C1 est}, and $Q\equiv 0$ on $\p \C_\theta$, 
		\begin{eqnarray*}
			0&\geq &Q_{\mu}(\eta_{0})\\
			&\geq &(h_{k\mu}\zeta_{k}+h_{k}\zeta_{k\mu})-\left(A+\frac{1}{2}M\right)\zeta_{\mu}-\cot\theta h_{\mu}\\
			&\geq& h_{\mu\mu}(\eta_{0})-\left(A+\frac{1}{2}M\right)-C,
		\end{eqnarray*}
		it yields  
		\begin{eqnarray}\label{sup estimate}
			\sup\limits_{\partial \C_{\theta}}h_{\mu\mu}\leq C+\frac{1}{2}M.
		\end{eqnarray}
		Similarly, we consider an auxiliary function as
		\begin{eqnarray*}
			\bar{Q}(\xi)\coloneqq \<\n h, \n \zeta\>+\left(\bar{A}+\frac{1}{2}M\right)\zeta(\xi)-\cot\theta h, \quad \xi\in \Omega_{\delta_{0}},
		\end{eqnarray*}
		where 	$\bar{A}>0$ is a positive constant. Adapting the similar argument as above, we get $$\bar{Q}(\xi)\leq 0 ~\text{ in }   \Omega_{\delta_{0}},$$ and further
		\begin{eqnarray}\label{inf estimate}
			\inf\limits_{\partial \C_{\theta}}h_{\mu\mu}\geq -C-\frac{1}{2}M.
		\end{eqnarray}
Then \eqref{sup estimate} and \eqref{inf estimate} yield
		\begin{eqnarray*}
			\sup\limits_{\partial \C_{\theta}}|h_{\mu\mu}|\leq C.
		\end{eqnarray*}
Together with Lemma  \ref{c2 to normal}, we obtain
		\begin{eqnarray*}
			\sup\limits_{\C_{\theta}}|\n^{2} h|\leq C.
		\end{eqnarray*}
		Thus, we complete the proof of Lemma \ref{c2 est}. 

\end{proof}

\begin{proof}[\textbf{Proof of Theorem \ref{priori estimate}}]
From Lemma \ref{c0 est}, Lemma \ref{C1 est} and Lemma \ref{c2 est},   we obtain
\begin{eqnarray*}
	\|h\|_{C^{2}(\C_\theta)}\leq C.
\end{eqnarray*}
Applying the theory of fully nonlinear second order uniformly elliptic equations with oblique derivative boundary condition (cf. \cite[Theorem 1.1]{LT} and \cite{LTU} etc.), then the  Schauder estimates and higher order estimates as in \eqref{sch est} follow, thus we complete the proof of Theorem \ref{priori estimate}.
 \end{proof}

\section{Convexity}\label{sec-convex}
The primary goal of this section is to obtain a strictly convex solution of Eq. \eqref{eq-k}. Based on the celebrated result of Guan-Ma \cite{GM} and with suitable assumptions on the function $f$, we show that the convex solution of Eq. \eqref{eq-k} is strictly convex.  To begin with, we establish the following key lemma.
\begin{lem}\label{lem-def}
  Let $\theta\in (0, \frac{\pi}{2})$ and  $\tilde{F}(W) \coloneqq \sigma_{k}(W)$.  Suppose that  $h\in C^{2}(\C_{\theta})$ is a  solution of Eq. \eqref{eq-k}, and assume that $W$  is positive semi-definite on $\C_\theta$.  Suppose there exists a positive constant $C_{0}$ such that for a fixed integer $l $ $(k\leq l\leq n-1)$, $\sigma_{l}(W)\geq C_{0}$ for all $\xi\in \C_{\theta}$. If $f\in \mathcal{L}_{-\frac{1}{k}}$  and $\n_{\mu}f\geq 0$ on $\partial \C_{\theta}$, denote $$\phi\coloneqq \sigma_{l+1}(W(\xi)),$$ then there exist constants $C_{1}, C_{2}$ depending only on $n, k, C_{0}, \|h\|_{C^{3}},$ and $\|f\|_{C^{2}}$, such that $\phi$ satisfies 
    \begin{eqnarray}
        & \tilde{F}^{ij}\phi_{ij}\leq C_{1}|\n \phi|+C_{2}\phi,  & \quad \rm{in} ~~ \C_{\theta} ,\label{Diff-ine}\\
        &\n_{\mu} \phi\geq 0, & \quad {\rm{on}}~\partial \C_{\theta}.\label{boundary ine}
    \end{eqnarray}
  \end{lem}
  \begin{proof}
      The differential inequality \eqref{Diff-ine} follows from \cite[Lemma~4.1]{GM} and $f\in \mathcal{L}_{-\frac{1}{k}}$. We just need to prove the boundary condition \eqref{boundary ine}.

     For any $\xi\in\partial \C_{\theta}$, according to \eqref{h1n}, we can choose an orthonormal frame $\{\{e_{\a}\}_{\a=1}^{n-1}, e_{n}=\mu\}$ around $\xi$ such that $\n^{2}h$ is diagonal at this point $\xi$. 
From $\n_{\mu}f\geq 0$ on $\partial\C_{\theta},$ we get
\begin{eqnarray}\label{one deri}
    0\leq \n_\mu f=\n_{\mu} \tilde F=\sum\limits_{\alpha=1}^{n-1}\tilde{F}^{\alpha\alpha}W_{\alpha\alpha n}+\tilde{F}^{nn}W_{nnn}, \quad \text{on}~\partial\C_{\theta}.
\end{eqnarray}
Denote $H(W)\coloneqq\sigma_{l+1}(W)$ and $H^{ij}\coloneqq\frac{\partial H}{\partial W_{ij}}$. Using the Gauss-Weingarten equation of $\p \C_\theta \subset \C_\theta$ and combining with \eqref{h1n}, we have 
		\begin{eqnarray*}
			\n_{\alpha \alpha n}h&=&(\n^3 h)(e_\a,e_\a,e_n)	=(\n_{e_\a}(\n^2 h))(e_\a,e_n)\\&=& \n_{e_\a}(\n^2h(e_\alpha,e_n))-\n^2h(\n_{e_\a}e_\a,e_n)-\n^2h(e_\a,\n_{e_\a}e_n)
			\\&=&\cot\theta h_{nn}-\cot\theta h_{\a\a}.
		\end{eqnarray*} 
In view of  the commutator formula \eqref{third comm}, we derive
\begin{eqnarray}\label{h11n}
	\n_{n\a\a}h=\n_{\a\a n}h-h_n=\cot\theta ( h_{nn}-  h_{\a\a}-  h), 
\end{eqnarray}together with $\n_n h=\cot\theta h$ on $\p \C_\theta$, it follows
\begin{eqnarray*}
    W_{\a\a n}= \cot\theta(h_{nn}-h_{\a\a})=\cot\theta(W_{nn}-W_{\a\a}), ~~1\leq \a\leq n-1.
\end{eqnarray*}
Combining this with \eqref{h11n} and \eqref{one deri}, 
\begin{eqnarray*}\n_{\mu} \phi&=&\sum\limits_{\alpha=1}^{n-1}H^{\alpha\alpha}W_{\alpha\alpha n}+H^{nn}W_{nnn}
    \geq \frac{1}{\tilde{F}^{nn}}\sum\limits_{\alpha=1}^{n-1}(H^{\alpha\alpha}\tilde{F}^{nn}-\tilde{F}^{\alpha\alpha}H^{nn})W_{\alpha\alpha n}\\
    &=&\frac{\cot\theta}{\tilde{F}^{nn}}\sum\limits_{\alpha=1}^{n-1}(W_{nn}-W_{\alpha\alpha})(H^{\alpha\alpha}\tilde{F}^{nn}-\tilde{F}^{\alpha\alpha}H^{nn}).
\end{eqnarray*}
We next conclude the proof by distinguishing the range of index $l\in \{1,2,\ldots, n-1 \}$.
\begin{enumerate}
    \item 
If $l=n-1$, by direct calculation, 
\begin{eqnarray*}
    H^{\alpha\alpha}\tilde{F}^{nn}-H^{nn}\tilde{F}^{\alpha\alpha}=(W_{nn}-W_{\alpha\alpha})\sigma_{n-2}(W|\alpha n)\sigma_{k-1}(W|\alpha n),
\end{eqnarray*}
this implies
\begin{eqnarray*}
\n_{\mu} \phi \geq \frac{\cot\theta}{\tilde{F}^{nn}}\sum\limits_{\alpha=1}^{n-1}(W_{nn}-W_{\alpha\alpha})^{2}\sigma_{n-2}(W|\alpha n)\sigma_{k-1}(W|\alpha n)\geq 0.
\end{eqnarray*}

\item If $2\leq k\leq l\leq n-2$, we obtain
\begin{eqnarray*}
    H^{\alpha\alpha}\tilde{F}^{nn}-H^{nn}\tilde{F}^{\alpha\alpha}=(W_{\alpha\alpha}-W_{nn}) \left(\sigma_{l}(W|\alpha n)\sigma_{k-2}(W|\alpha n)-\sigma_{l-1}(W|\alpha n)\sigma_{k-1}(W|\alpha n)\right),
\end{eqnarray*}
together with Lemma \ref{prop2.3}, it follows
\begin{eqnarray*}
  \n_{\mu} \phi&\geq &\frac{\cot\theta}{\tilde{F}^{nn}}\sum\limits_{\alpha=1}^{n-1}(W_{nn}-W_{\alpha\alpha})^{2}\cdot \left(\sigma_{l-1}(W|\alpha n)\sigma_{k-1}(W|\alpha n)-\sigma_{l}(W|\alpha n)\sigma_{k-2}(W|\alpha n)\right)\\
    &\geq &\frac{(n-1)(l-k+1)}{(k-1)(n-l-1)}\frac{\cot\theta }{\tilde{F}^{nn}}\sum\limits_{\alpha=1}^{n}(W_{nn}-W_{\alpha\alpha})^{2}\left(\sigma_{l-1}(W|\alpha n)\sigma_{k-1}(W|\alpha n)\right)\\
    &\geq& 0.
\end{eqnarray*}
\item If $1=k\leq l\leq n-2$, using  Lemma \ref{prop2.2}~(4), we conclude that 
\begin{eqnarray*}
\n_{\mu} \phi =\cot\theta \sum\limits_{\alpha=1}^{n-1}(W_{nn}-W_{\alpha\alpha})(H^{\alpha\alpha}-H^{nn})\geq 0.
\end{eqnarray*}
\end{enumerate}
This completes the proof.
 \end{proof}

As an application of Lemma \ref{lem-def}, we establish the following theorem.
\begin{thm}\label{convex thm}
 Let $\theta\in(0, \frac \pi 2)$.   Suppose that $h$ is a non-negative, convex solution of Eq. \eqref{eq-k}. If $f$ satisfies the assumptions in Lemma \ref{lem-def}, then $W$ is positive definite everywhere on $\C_{\theta}$.
\end{thm}

\begin{proof} We argue by contradiction. If $W$ does not have full rank at some point, say at  $\xi_{0}\in \C_{\theta}$, then there exists an integer $l$ $(k\leq l\leq n-1)$ such that $\phi(\xi_{0})=\sigma_{l+1}(W(\xi_{0}))=0$ and $\sigma_{l}(W(\xi))>0 $ for all $\xi\in \C_{\theta}$. From  Lemma \ref{lem-def}, the maximum principle  and Hopf lemma imply   $$\phi=\sigma_{l+1}(W)\equiv 0.$$ Since $h\geq 0$ in  $\C_{\theta}$ and the Minkowski formulas (cf. \cite[Corollary~2.10]{MWWX}),
    \begin{eqnarray*}
      (n-l)\int_{\C_{\theta}}h(\xi)\sigma_{l}(W)d\H^{n}=(l+1)\int_{\C_{\theta}}\ell(\xi)\sigma_{l+1}(W)d\H^{n},
    \end{eqnarray*}
we conclude that $h\equiv 0$ in $\C_{\theta}$, which is a contradiction to Eq. \eqref{eq-k}. Hence, the conclusion follows.
\end{proof}

\section{Proof of Theorem \ref{main theorem}}
In this section, we use the standard continuity method to complete the proof of Theorem \ref{main theorem}. Since  $f$ is connected to $1$ in $\mathcal{L}_{-\frac{1}{k}}$, there is a continuous path  $p(t, \xi)$ satisfying $p(0, \xi)=1, p(1, \xi)=f(\xi)$, and $p(t, \xi)\in \mathcal{L}_{-\frac{1}{k}}$ for any $0\leq t\leq 1$. We define the set
	\begin{eqnarray*}
		\mathcal{S} \coloneqq  \left\{ t\in[0,1] \mid  \sigma_{k}(\n^2 h+h\s)=p(\xi, t)   \text{ in } \C_\theta, \text{ and } \n_\mu h=\cot\theta h \text{ on } \p \C_\theta  \text{ is solvable}  \right\}.
	\end{eqnarray*}It is easy to see that $0\in \mathcal{S}$, since the function $h=\binom{n}{k}^{-\frac{1}{k}}\ell(\xi)$ is a solution. Next, we will apply the implicit function theorem to show the openness of $\mathcal{S}$, while the closedness follows from Theorem \ref{priori estimate} and Theorem \ref{convex thm}. 

 Let $$W^{m,2}_{N}(\C_\theta)\coloneqq \left\{ h\in W^{m,2}(\C_\theta) \mid  ~\n_\mu h=\cot\theta h \text{ on } \p \C_\theta\right\},$$  be the Sobolev space, and let $m\in \NN$ sufficiently large such that $W^{m,2}_{N}(\C_\theta)\subset C^4_{N}(\C_\theta)$, where $$C^{4}_{N}(\C_\theta)\coloneqq \{ h\in C^4(\C_\theta) \mid ~ \n_\mu h=\cot\theta h  \text{ on } \p \C_\theta\},$$ 
and the nonlinear operator $G$ in $W^{m,2}_{N}(\C_{\theta}) $ as
	\begin{eqnarray*}
		G:&W^{m,2}_{N}(\C_{\theta}) &\to \RR\\&
		h &\mapsto \sigma_{k}(\n^2 h +h\s).
	\end{eqnarray*}
  \begin{prop}\label{pro open}
     The set $\mathcal{S}$ is open in $[0,1]$.
 \end{prop}
 In order to prove the openness, we notice that the linearized operator $L_h$ of $G$ at $h$ is
	\begin{eqnarray*}
		L_h(v)\coloneqq \frac{\p \sigma_{k}(W)}{\p W_{ij}} (v_{ij}+v\delta_{ij}).
	\end{eqnarray*}
 It is easy to observe that $\<\xi,E_\a\>\in C^4_{N}(\C_\theta)$ and satisfies  \begin{eqnarray*}
	L_h( \<\xi,E_\a\>)=0, ~~\text{ for all } 1\leq \a\leq n.\end{eqnarray*} Namely, the function 	$ \<\xi,E_\a\>$ $(\a=1,2,\cdots, n)$ belongs to the kernel of $L_h$.
The following proposition shows that $L_h$ is a self-adjoint operator.
	\begin{lem}\label{self adj of L}
		For any $h,v,w\in W^{m,2}_{N}(\C_{\theta})$, there holds 
		\begin{eqnarray*}
			\int_{\C_\theta} w L_h vd\H^n=	\int_{\C_\theta} v L_h wd\H^n.
		\end{eqnarray*}
	\end{lem}
\begin{proof}
Since  $\n_{\mu}v=\cot\theta v$ on $\partial\C_{\theta}$, and integrating by parts twice, we have
		\begin{eqnarray*}
			\int_{\C_\theta} w L_h vd\H^n&=&\int_{\C_\theta} w  \frac{\p \sigma_{k}(W)}{\p W_{ij}} (v_{ij}+v\delta_{ij}) d\H^n
			\\&=& \int_{\p \C_\theta} w  \frac{\p \sigma_{k}(W)}{\p W_{ij}}  v_i\mu^j d\H^{n-1} + \int_{\C_\theta}\frac{\p \sigma_{k}(W)}{\p W_{ij}} \left( vw   \delta_{ij} -w_j v_i \right) d\H^n
			\\&=& \int_{\p \C_\theta} \cot\theta wv  \frac{\p \sigma_{k}(W)}{\p W_{nn}} d\H^{n-1} + \int_{\C_\theta}\frac{\p \sigma_{k}(W)}{\p W_{ij}} \left( vw   \delta_{ij} -w_j v_i \right) d\H^n,
		\end{eqnarray*}which is symmetric with respect to $w$ and $v$. Hence, we complete the proof of Lemma \ref{self adj of L}.
	\end{proof}
	
	The following Corollary is an immediate consequence of Lemma \ref{self adj of L}.
	\begin{cor}
		For any $h\in W^{m,2}_{N}(\C_{\theta})$, there holds
		\begin{eqnarray*}
			\int_{\C_\theta} \<\xi,E_\a\> \sigma_{k}(W) d\H^n=0,~ 1\leq \a\leq n.
		\end{eqnarray*}
	\end{cor}
	
	\begin{proof}
		Since  $w_\a\coloneqq \<\xi,E_\alpha\>$  ($1\leq \a\leq n)$ belongs to the kernel of $L_h$, Lemma \ref{self adj of L} implies
		$$\int_{ \C_\theta} w_\a L_h(h) d\H^{n}= \int_{ \C_\theta}h L_h(w_\a)d\H^{n}=0.$$
        Together with Lemma \ref{prop2.2} (2), the assertion follows.
	\end{proof}
 Now, by following the argument in \cite[Section~2]{GMTZ}, we will show that the kernel of the linearized operator $L_h$ consists only of functions in the span of $\left\{\<\xi, E_{1}\>, \cdots, \<\xi, E_{n}\> \right\}$.
	\begin{lem}\label{kernel of L}
		Let $v\in W^{m,2}_{N}(\C_{\theta})$ satisfy $L_h(v)=0$, and $W=(h_{ij}+h\s_{ij})>0$. Then 
		\begin{eqnarray*}
			v=\sum_{\a=1}^n a_\a \<\xi, E_{\alpha}\>,
		\end{eqnarray*}for some constants $a_1,a_2,\ldots, a_n \in \RR$.
	\end{lem}
 \begin{proof}
  Let $\{e_i\}_{i=1}^n$ be a local orthonormal frame field on $\C_\theta$ and $e_{n+1}\coloneqq \xi+\cos\theta E_{n+1}$  the unit outward normal of $\C_\theta\subset \ol{\RR^{n+1}_+}$ such that $\{e_i\}_{i=1}^{n+1}$ forms a positive-oriented orthonormal frame with respect to the standard Euclidean metric in $\ol{\RR^{n+1}_+}$.  Let $\{\omega^{1},\cdots, \omega^{n+1}\}$ be the dual $1$-form of $\{e_{1},\cdots, e_{n}, e_{n+1}\}$. For any $v\in C^{4}_{N}(\C_{\theta})$, we consider the vector-valued functions 
     $Z=\sum\limits_{i=1}^{n}v_{i}e_{i}+ve_{n+1}$, then $v=\<Z, e_{n+1}\>$ and 
     \begin{eqnarray*}
         dZ=(v_{ij}+v\delta_{ij})e_{i}\omega^{j},
     \end{eqnarray*}
     Let $X= \sum_{i=1}^n h_{i}e_{i}+h e_{n+1}$, then 
     \begin{eqnarray*}
         dX=(h_{ij}+h\delta_{ij})e_{i}\omega^{j}.
     \end{eqnarray*}
    Denote $A[v]\coloneqq (v_{ij}+v \delta_{ij})$, and it is easy to verify  $A[\ell]$ is the identity matrix $I$. Moreover, for symmetric matrix $W^{i}~(1\leq i\leq k)$, we recall that 
    \begin{eqnarray*}
        \sigma_{k}(W^{1}, W^{2},\cdots, W^{k})\coloneqq \frac{1}{k!}\delta^{i_{1}i_{2}\cdots i_{k}}_{j_{1}j_{2}\cdots j_{k}}(W^{1})_{i_{1}j_{1}}(W^{2})_{i_{2}j_{2}}\cdots (W^{k})_{i_{k}j_{k}},
    \end{eqnarray*}
    and define the $n$-form,
    \begin{eqnarray*}
    \O (\underbrace{v,\cdots, v}_{l~\rm{copies}}, \underbrace{ h,\cdots, h,}_{(k-l+1)~\rm{copies}}\underbrace{\ell, \cdots, \ell}_{(n-k)~\rm{copies}}) 
        \coloneqq v \sigma_{k}(\underbrace{A[v], \cdots, A[v]}_{(l-1)~\rm{copies}}, \underbrace{W,\cdots, W}_{(k-l+1)~\rm{copies}})d\H^{n}.
    \end{eqnarray*}
From \cite[Proposition~2.9]{MWWX}, we have 
    \begin{eqnarray*}
    \int_{\C_{\theta}} \O (\underbrace{v,\cdots, v}_{l~\rm{copies}}, \underbrace{ h,\cdots, h}_{(k-l+1)~\rm{copies}}, \underbrace{\ell, \cdots, \ell}_{(n-k)~\rm{copies}})=\int_{\C_{\theta}}\O (h, \underbrace{v,\cdots, v}_{l~\rm{copies}}, \underbrace{ h,\cdots, h}_{(k-l)~\rm{copies}}, \underbrace{\ell, \cdots, \ell}_{(n-k)~\rm{copies}}).
    \end{eqnarray*}    
    If $v$ is in kernel of $L_{h}$, i.e., $L_{h}v=0$, then 
    \begin{eqnarray*}
        0=\int_{\C_{\theta}}v L_{h}(v)d\H^{n}&=&\int_{\C_{\theta}}v \frac{\partial \sigma_{k}(W)}{\partial W_{ij}}(v_{ij}+v\delta_{ij})d\H^{n}\notag\\
        &=&\int_{\C_{\theta}} \O(v, v, \underbrace {h,\cdots, h}_{(k-1)~\rm{copies}}, \underbrace{\ell, \cdots, \ell}_{(n-k)~\rm{copies}})\notag \\
        &=&\int_{\C_{\theta}}h \sigma_{k}(A[v], A[v], \underbrace {W,\cdots, W}_{(k-2)~\rm{copies}})d\H^{n}.
    \end{eqnarray*}
    Together with  the same argument in the proof of  \cite[Proposition~3.4]{GMTZ}, we conclude that $A[v]=0$, then 
    \begin{eqnarray}\label{linear sol}
			v=\sum_{\a=1}^{n} a_\a \<\xi, E_{\alpha}\>+a_{n+1} \<\xi+\cos\theta E_{n+1}, E_{n+1}\>,
		\end{eqnarray}for some constants $\{a_i\}_{i=1}^{n+1}\subset \R$.
    Since $$\n_{\mu}v=\cot\theta v \text{  on } \partial\C_{\theta},$$ we find that $v=\sum\limits_{\a=1}^n a_\a \<\xi, E_{\alpha}\>$ (cf. \cite[Proof of Lemma 4.3]{MWW-AIM}). This completes the proof.
\end{proof}

We are now in a position to prove Proposition \ref{pro open} and Theorem \ref{main theorem}.
\begin{proof}[\textbf{Proof of Proposition \ref{pro open}}]
For convex solution $h\in W^{m,2}_{N}(\C_{\theta}) $, by Lemma \ref{self adj of L} and Lemma \ref{kernel of L}, we know that $L_{h}$ is a self-adjoint operator satisfying  
$$\text{Ker}(L_{h})=\text{span}\{\<\xi, E_{1}\>,\cdots, \<\xi, E_{n}\>\},$$ and in turn 
\begin{eqnarray*}
    \text{Range} (L_{h})=\text{Ker}(L_{h}^{\ast})^\perp=\left(\text{span}\{\<\xi, E_{1}\>,\cdots, \<\xi, E_{n}\>\}\right)^{\perp}.
\end{eqnarray*}
Then for any $f$ satisfies $\int_{\C_{\theta}}\<\xi, E_{\alpha}\>f(\xi)d\H^{n}=0$, for $\alpha=1,\cdots, n$, we have $f\in\text{Range}(L_{h}),$ this implies the operator $L_{h}$ is surjective. The implicit function theorem indicates $G$ is locally invertible around $h$, therefore $\mathcal{S}$ is open in $[0,1]$.
\end{proof}

\begin{proof}[\textbf{Proof of Theorem \ref{main theorem}}]
Theorem \ref{convex thm}, Theorem \ref{priori estimate}, and Proposition \ref{pro open} imply that the set $\mathcal{S}$ is nonempty, closed, and open. Then we conclude that  $\mathcal{S}=[0, 1]$. In particular, by Theorem \ref{convex thm}, the solution $h$ of Eq. \eqref{eq-k} is strictly convex. The uniqueness can be shown similarly by applying the argument as in \cite[Theorem~2]{Chern}  and \cite[Theorem~3.2]{GMTZ}, and hence we omit it here for brevity.
\end{proof}

\ 

\bigskip

 {\footnotesize

\ 
 
\noindent\textbf{Acknowledgment:} X. M. was supported by  the National Key R$\&$D Program of China 2020YFA0712800  and the Postdoctoral Fellowship Program of CPSF under Grant Number  2025T180843 and 2025M773082.  L. W. was partially supported by CRM De Giorgi of Scuola Normale Superiore and PRIN project 2022E9CF89 of the University of Pisa.
  
\medskip

\ 

\noindent\textbf{Data availability:} No datasets were generated or analysed during the current study.
 
}

\bigskip

\noindent\textbf{\large Declarations} 

 \
 
 {\footnotesize

\noindent\textbf{Conflict of interest:} The authors declare that they have no conflict of interest.
\medskip

\noindent\textbf{Ethics approval:}  The authors declare that they adhere to the ethical standards followed by the journal.
\medskip

 }

\printbibliography

\end{document}